\documentclass[11pt]{amsart}

\usepackage{amsmath, amsthm, amscd, amsfonts, amssymb, enumerate, verbatim, newlfont, calc, graphicx, color}
\usepackage[bookmarksnumbered, colorlinks, plainpages]{hyperref}
\usepackage{tikz}
 \usepackage{doi}

\newtheorem{theorem}{Theorem}[section]       
\newtheorem{question}[theorem]{Question}    
 
\newtheorem{lemma}[theorem]{Lemma}
\newtheorem{proposition}[theorem]{Proposition}

\newtheorem{definition}[theorem]{Definition}

\usepackage[normalem]{ulem}
 \usepackage{doi}

 \def\N{{\mathcal N}}

 \def\C{{\mathcal C}}

\def\H{{\mathcal H}}

 \def\ab{{\bold a}}
 \def\bb{{\bold b}}

 \def\cb{{\bold c}}

  \def\1b{{\bold 1}}
 \def\opn#1#2{\def#1{\operatorname{#2}}} 
 \opn\chara{char} \opn\length{\ell} \opn\pd{pd} \opn\rk{rk}
 \opn\projdim{proj\,dim} \opn\injdim{inj\,dim} \opn\rank{rank}
 \opn\depth{depth} \opn\grade{grade} \opn\height{height}
 \opn\embdim{emb\,dim} \opn\codim{codim}
 
 \opn\Tr{Tr} \opn\bigrank{big\,rank}
 \opn\superheight{superheight}\opn\lcm{lcm}
 \opn\trdeg{tr\,deg}
 \opn\reg{reg} \opn\lreg{lreg} \opn\ini{in} \opn\lpd{lpd}
 \opn\size{size} \opn\sdepth{sdepth}
 \opn\link{link}\opn\fdepth{fdepth}\opn\lex{lex}
 \opn\div{div} \opn\Div{Div} \opn\cl{cl} \opn\Cl{Cl}
 \opn\gr{gr}\opn\bight{bight}
 
 \opn\Spec{Spec} \opn\Supp{Supp} \opn\supp{supp} \opn\Sing{Sing}
 \opn\Ass{Ass} \opn\Min{Min}\opn\Mon{Mon}
 \opn\Ann{Ann} \opn\Rad{Rad} \opn\Soc{Soc}
 \opn\Im{Im} \opn\Ker{Ker} \opn\Coker{Coker} \opn\Am{Am}
 \opn\Hom{Hom} \opn\Tor{Tor} \opn\Ext{Ext} \opn\End{End}
 \opn\Aut{Aut} \opn\id{id}
 
 \opn\nat{nat}
 \opn\pff{pf}
 \opn\Pf{Pf} \opn\GL{GL} \opn\SL{SL} \opn\mod{mod} \opn\ord{ord}
 \opn\Gin{Gin} \opn\Hilb{Hilb}\opn\sort{sort}
 \opn\aff{aff} \opn
 \con{conv} \opn\relint{relint} \opn\st{st}
 \opn\lk{lk} \opn\cn{cn} \opn\core{core} \opn\vol{vol}
 \opn\link{link} \opn\star{star}\opn\lex{lex}\opn\set{set}
 \opn\gr{gr}
 
 \def\pot#1#2{#1[\kern-0.28ex[#2]\kern-0.28ex]}

 \opn\dirlim{\underrightarrow{\lim}}
 \opn\inivlim{\underleftarrow{\lim}}
 \def\Implies{\ifmmode\Longrightarrow \else
         \unskip${}\Longrightarrow{}$\ignorespaces\fi}
 \def\implies{\ifmmode\Rightarrow \else
         \unskip${}\Rightarrow{}$\ignorespaces\fi}
 \def\iff{\ifmmode\Longleftrightarrow \else
         \unskip${}\Longleftrightarrow{}$\ignorespaces\fi}

 \let\:=\colon
 \newtheorem{Theorem}{Theorem}[section]

 \newtheorem{Proposition}[Theorem]{Proposition}
 \newtheorem{Remark}[Theorem]{Remark}

 \newtheorem{Definition}[Theorem]{Definition}

 \let\epsilon\varepsilon
 \let\kappa=\varkappa
 \def\qed{\ifhmode\textqed\fi
       \ifmmode\ifinner\quad\qedsymbol\else\dispqed\fi\fi}
 \def\textqed{\unskip\nobreak\penalty50
        \hskip2em\hbox{}\nobreak\hfil\qedsymbol
        \parfillskip=0pt \finalhyphendemerits=0}
 \def\dispqed{\rlap{\qquad\qedsymbol}}
 
 \opn\dis{dis}
 \def\pnt{{\raise0.5mm\hbox{\large\bf.}}}
 
 \opn\Lex{Lex}

\usepackage{mathtools}

\usepackage{pstricks}
\usepackage{epsfig}
\usepackage{pst-grad}
\usepackage{pst-plot}
\usepackage{subfig}

\begin{document}
\author[M. Nasernejad and  A. A. Qureshi]{Mehrdad ~ Nasernejad$^{1}$ and   Ayesha Asloob Qureshi$^{2,*}$}
\title[Neighborhood hypergraphs and their transversal hypergraphs] {Algebraic implications of neighborhood hypergraphs and their transversal hypergraphs}

\thanks{$^*$Corresponding author}

\thanks{Ayesha Asloob Qureshi is supported by The Scientific and Technological Research Council of Turkey - T\"UBITAK (Grant No: 122F128), and acknowledge support from the ICTP through the Associates Programme (2019-2023).}

\thanks{E-mail addresses:  m\_{n}asernejad@yahoo.com  and  aqureshi@sabanciuniv.edu}  

\maketitle

\begin{center}
{\it
$^1$Univ. Artois, UR 2462, Laboratoire de Math\'{e}matique de  Lens (LML),\\
 F-62300 Lens, France  \\
 $^2$Sabanci University, Faculty of Engineering and Natural Sciences, \\
Orta Mahalle, Tuzla 34956, Istanbul, Turkey
}
\end{center}

\begin{abstract}
In this paper, we unfold balanced and totally balanced neighborhood hypergraphs to discover new classes of normally torsion-free monomial ideals. As a consequence, we establish that the closed neighborhood ideals and the dominating ideals of strongly chordal graphs are normally torsion-free. We discuss the stable sets of associated primes of the dominating ideals of cycles and characterize all the cycles with normally torsion-free dominating ideals.
\end{abstract}
\vspace{0.4cm}


\section*{Introduction}

Let $R$ be a commutative Noetherian ring and let $\mathrm{Ass}_R(R/I)$ be the set of all prime ideals associated to an ideal $I$ of $R$. An ideal $I \subset R$ is called {\em normally torsion-free} if $\mathrm{Ass}_R(R/I^k) \subseteq \mathrm{Ass}_R(R/I)$  for all $k\geq 1$.
 It is a known fact that if $I$ is a square-free monomial ideal in a polynomial ring $S=K[x_1, \ldots, x_n]$ over a field $K$, then $I$ is normally torsion-free if and only if the ordinary powers of $I$ coincide with the symbolic powers of $I$, for example see \cite[Theorem 1.4.6]{HH1}. Some of the well-known classes of normally torsion-free monomial ideals originate from graph theory. 
 Simis, Vasconcelos, and Villarreal showed in \cite{SVV}  that  a finite simple graph is bipartite if and only if its edge ideal is normally torsion-free.  In \cite{GRV}, Gitler, Reyes, and Villarreal showed that the cover ideals of bipartite graphs are normally torsion-free. 

Let $G$ be a simple graph with a vertex set $V(G)$. Given a vertex $v \in G$, the closed neighborhood of $v$ is a subset of $V(G)$ consisting of $v$ and all the vertices that are adjacent to $v$. A dominating set of $G$ is a subset of $V(G)$ that intersects  the closed neighborhoods of all $v \in V(G)$. In the 1960s, Berge and Ore mathematically formulated the concept of domination in graphs and it has been  studied widely by many researchers due to its far-reaching and growing applications in various fields including computer sciences, operations research, linear algebra, and optimization. For a given graph $G$, characterizing all the dominating sets of $G$ and finding the smallest size of a dominating set are among the central problems in the domination of graphs. We refer to \cite{HHS} for further concepts related to the domination in graphs. The definitions of closed neighborhood ideals and dominating ideals were introduced in \cite{SM}. The closed neighborhood ideal of $G$, denoted by $NI(G)$, is the square-free monomial ideal generated by the monomials that correspond to the closed neighborhoods of vertices of $G$ and the dominating ideal of $G$, denoted by $DI(G)$, is the square-free monomial ideal generated by the monomials that correspond to the dominating sets of $G$. In \cite{NQBM}, the authors discussed some classes of graphs whose closed neighborhood ideals and dominating ideals are normally torsion-free. In this paper, our motivation is to extend the work carried out in \cite{NQBM} and widely expand the classes of graphs that give normally torsion-free closed neighborhood ideals and dominating ideals.

The breakdown of the content of this paper is as follows: Section~\ref{prem} contains all the required ingredients from commutative algebra, graph theory,  and hypergraph theory. In our work, the main reference for the definitions and notions in hypergraph theory is  Berge's book  \cite{B}. In Section~\ref{section2}, we review the literature in hypergraph theory and discover some interesting classes of Mengerian hypergraphs, and in particular, we focus on the neighborhood hypergraphs of simple graphs. Normally torsion-free square-free monomial ideals  are closely related to Mengerian hypergraphs. Let $\mathcal{H}$ be a simple hypergraph. The edge ideal of $\mathcal{H}$, denoted by $I(\mathcal{H})$, is the ideal generated by the monomials corresponding to the edges of $\mathcal{H}$. A hypergraph $\mathcal{H}$ is called {\em Mengerian} if it satisfies a certain  min-max equation, which is known as the Mengerian property in hypergraph theory or  has the max-flow min-cut property in integer programming. Algebraically, it is equivalent to $I(\mathcal{H})$ being normally torsion-free, as illustrated in \cite[Corollary 10.3.15]{HH1} and \cite[Corollary 3.14]{GVV}.  The balanced hypergraphs and the totally balanced hypergraphs are among the most discussed classes of Mengerian hypergraphs. In simple words, balanced hypergraphs are those hypergraphs that avoid ``special" (hypergraphic) cycle of odd length, and the totally balanced hypergraphs are those hypergraphs that avoid ``special" (hypergraphic) cycle of any length. Their formal definitions are stated in Definition~\ref{hypercycle}. Using the results obtained in \cite{BDS}, as a consequence in Theorem \ref{trampoline}, we establish that if $G$ is a chordal graph with no induced odd incomplete trampoline, then the closed neighborhood ideal and the dominating ideal of $G$ is normally torsion-free. 

Totally balanced hypergraphs are of particular interest because they generalize the notion of forests (in the sense of graphs). Moreover, it is shown in \cite[Theorem 2.3]{HHTZ} that totally balanced hypergraphs can be identified as simplicial forests. Simplicial forests  were introduced by Faridi in \cite{Far1} and further explored in \cite{Far2}. In \cite{Fa}, Farber introduced strongly chordal graphs and proved that a graph $G$ is strongly chordal if and only if the neighborhood hypergraph of $G$ is totally balanced. Rephrasing this in the algebraic language, we conclude in Theorem~\ref{stronglychordal} that the closed neighborhood ideals  of strongly chordal graphs are normally torsion-free and sequentially Cohen-Macaulay, and their dominating ideals are normally torsion-free and componentwise linear. 

 Next, we examine the closed neighborhood ideals and dominating ideals  of cycles. The closed neighborhood ideals of cycles can be identified as $3$-path ideals (or cubic path ideals as in \cite{AB}) of cycles. To deal with $DI(C_n)$, we first review the case of $t$-path ideals $P_n$, in Section~\ref{tpath}. Here $P_n$ denotes the path on $n$ vertices. Let $\H_t(P_n)$ be the hypergraph whose edges are the paths of length $t-1$ in $P_n$ with $n \geq t$. Then $\H_t(P_n)$ is a simplicial forest as shown in \cite[Theorem 2.7]{JT}  and once again by \cite[Theorem 3.2]{HHTZ}, and we have   $\H_t(P_n)$ is totally balanced. This discussion yields a combinatorial proof of normally torsion-freeness of $t$-path ideals of path graphs. We give algebraic proofs of these statements in  Theorem~\ref{Path.NTF}.

 Section~\ref{DIofcycles} is dedicated to the dominating ideals of cycles. It is shown in \cite{AB} that the 3-path ideal of a cycle $C_n$ is normally torsion-free if and only if $n\in \{3,6,9\}$. The 3-path ideal of $C_n$ is the same as $NI(C_n)$. In the first main result of Section~\ref{DIofcycles}, we prove that 

\textbf{Theorem \ref{symbolic-cycles}.} \;
{\em Let $C_n$ be a cycle with $n\geq 3$.  Then $DI(C_n)^t=DI(C_n)^{(t)}$ for all $t \geq 1$ if and only if $n\in\{3,6,9\}$. In particular, $DI(C_n)$ is normally torsion-free if and only if $n\in\{3,6,9\}$. }
\newline

This result shows that $NI(C_n)$ is normally torsion-free if and only if $DI(C_n)$ is  normally torsion-free. The stable sets of associated primes of cover ideals of odd cycles are studied in \cite{NKA}. Inspired by \cite[Proposition 3.6]{NKA}, in Theorem~\ref{NNTF-DI(CN)}, we describe the stable sets  of associated primes of dominating ideals of cycles and prove also that $DI(C_{n})$ have the persistence property. We conclude this section
 with a short argument on the relation between  prime monomial ideals  associated to  $\mathrm{Ass}(DI(G)^s)$ for some $s \geq 1$, and prime monomial ideals associated to  $\mathrm{Ass}(DI(G_\mathfrak{p})^s)$, where  $G_\mathfrak{p}$ stands for the induced graph on $\mathfrak{p}$. Specially, we are interested in knowing  that if   $\mathfrak{p}=(x_{i_1}, \ldots, x_{i_r})\in \mathrm{Ass}(DI(G)^s)$ for some $s\geq 1$, then can we deduce that $G_\mathfrak{p}$ is connected?  We remain it as an open question, see Question \ref{OPEN-QUESTION}.


 \section{Required terminologies}\label{prem}
 \subsection{Algebraic background}
In this section, recall some definitions and notions from commutative algebra. Let  $R$ be  a commutative Noetherian ring and $I$ be   an ideal of $R$. A prime ideal $\mathfrak{p}\subset  R$ is an {\it associated prime} of $I$ if there exists an element $v$ in $R$ such that $\mathfrak{p}=(I:_R v)$, where $(I:_R v)=\{r\in R \mid  rv\in I\}$. The  {\it set of associated primes} of $I$, denoted by  $\mathrm{Ass}_R(R/I)$, is the set of all prime ideals associated to  $I$. 
The minimal members of $\mathrm{Ass}_R(R/I)$ are called the {\it minimal} primes of $I$, and $\mathrm{Min}(I)$ denotes the set of minimal prime ideals of $I$. Moreover, the associated primes of $I$ which are not minimal are called the {\it embedded} primes of $I$. If $I$ is a square-free monomial ideal, then $\mathrm{Ass}_R(R/I)=\mathrm{Min}(I)$, for example see \cite[Corollary 1.3.6]{HH1}.   Let $I$  be an ideal  of $R$  and  $\mathfrak{p}_1, \ldots, \mathfrak{p}_r$  be   the minimal primes of $I$. When there is no confusion about the underlying ring, we will denote the set of associated primes of $I$ simply by $\mathrm{Ass}(R/I)$ or $\mathrm{Ass}(I)$.    
Given an integer $n \geq 1$, the {\it $n$-th symbolic power} of $I$  is defined to be the ideal 
$$I^{(n)} = \mathfrak{q}_1  \cap \cdots \cap  \mathfrak{q}_r,$$
where $\mathfrak{q}_i$  is the primary component of $I^n$  corresponding to $\mathfrak{p}_i$. 

The ideal $I \subset R$ is said to have the {\it persistence property} if $\mathrm{Ass}(R/I^k)\subseteq \mathrm{Ass}(R/I^{k+1})$ for all positive integers $k$. Brodmann in \cite{Br} showed that in a Noetherian ring $R$ the set of associated prime ideals $\mathrm{Ass}(R/I^s)$ for the powers of an ideal $I$ stabilizes for $s \gg 0$. In other words, there exists the  minimum integer $k>0$ such that $\mathrm{Ass}(R/I^s)=\mathrm{Ass}(R/I^{s+1})$ for all $s\geq k$. This stable set of associated prime ideals is denoted by $\mathrm{Ass}^{\infty}(I)$.

\begin{definition}{\em
An ideal $I$ is called {\it normally torsion-free} if  $\mathrm{Ass}(R/I^k)\subseteq \mathrm{Ass}(R/I)$, for all $k\geq 1$. If $I$ is a square-free monomial ideal, then $I$ is normally torsion-free if and only if $I^k=I^{(k)}$   for all $k \geq 1$  (cf.  \cite[Theorem 1.4.6]{HH1}). 
}
 \end{definition}


\begin{definition}{\em Let $R$ be a unitary commutative ring and $I$ an ideal in $R$. An element $f\in R$ is {\it integral} over $I$, if there exists an equation 
 $$f^k+c_1f^{k-1}+\cdots +c_{k-1}f+c_k=0 ~~\mathrm{with} ~~ c_i\in I^i.$$
 The set of elements $\overline{I}$ in $R$ which are integral over $I$ is the 
 {\it integral closure} of $I$. The ideal $I$ is {\it integrally closed}, if $I=\overline{I}$, and $I$ is {\it normal} if all powers of $I$ are integrally closed, refer to \cite{HH1} for more information. The notion of integrality for a monomial ideal $I$ can be described in a simpler way as following: a monomial $u \in  R=K[x_1, \ldots, x_n]$ is integral over $I\subset R$ if and only if there exists an integer $k$ such that  $u^k \in I^k$, see \cite[Theorem 1.4.2]{HH1}.
 }
\end{definition}


Let $R=K[x_1, \ldots, x_n]$ be a polynomial ring over a field $K$. An ideal $I\subset R$ is called a monomial ideal if it admits a generating set consisting of monomials. Furthermore, $I$ is said to be a square-free monomial ideal if it is generated by square-free monomials. Throughout the following text, the unique minimal generating set of a monomial ideal $I$ will be denoted by $\mathcal{G}(I)$. The {\em support} of a monomial $u$, denoted by $\mathrm{supp}(u)$, is the set of variables that divide $u$. Moreover, for a monomial ideal $I$, we set $\mathrm{supp}(I)=\bigcup_{u \in \mathcal{G}(I)}\mathrm{supp}(u)$. Given a square-free monomial ideal $I\subset R$,  the {\it Alexander dual}  of  $I$, denoted by $I^\vee$,  is given by 
 $$I^\vee= \bigcap_{u\in \mathcal{G}(I)} (x_i~:~ x_i \in \mathrm{supp}(u)).$$
 

\subsection{Graph theory}
Let $G$ be a finite simple undirected graph with the vertex set $V(G)$ and the edge set $E(G)$. The degree of a vertex $v \in V(G)$ is denoted by $\deg(v)$ and it represents the number of vertices adjacent to $v$. A {\em subgraph} of $G$ is a graph $H$ with $V(H) \subseteq V(G)$ and $E(G)\subseteq E(H)$. A subgraph $H$ of $G$ is called an {\em induced} subgraph if it contains all the edges of $G$ that have both vertices in $H$. 

A set $T\subseteq V(G)$ is called a {\em vertex cover} of $G$ if it intersects every edge of $G$ non-trivially. A vertex cover is called minimal if it does not properly contain any other vertex cover of $G$. For each vertex $v \in V(G)$, the {\em closed neighborhood} of $v$ in $G$ is defined as follows: 
\[
N_G[v]=\{u \in V(G) : \{u,v\} \in E(G)\} \cup \{v\}.
\]
When there is no confusion about the underlying graph, we will denote $N_G[v]$ simply by $N[v]$. A subset $S\subseteq V(G)$ is called a  {\em dominating set} of $G$ if $S \cap N[v] \neq \emptyset$   for all $v\in V(G)$. A dominating set is called {\it minimal} if it does not properly contain any other dominating set of $G$. A minimum dominating set of $G$ is a minimal dominating set with the smallest size.  The {\em dominating number} of $G$, denoted by $\gamma(G)$, is  the size of its minimum dominating set, that is,
\[
\gamma(G)=\mathrm{min} \{|S|: S \text{ is a minimal dominating set
of } G\}.
\]

The dominating sets and domination numbers of graphs are
 well-studied topics in graph theory.  Further   information on dominations in graphs can be found in \cite{HHS}. 
 

 In general,  square-free monomial ideals  can be associated with graphs in many different ways. We recall some commonly known definitions in this context. Let $V(G)=\{x_1,x_2,\ldots, x_n\}$ and $R=K[x_1, \ldots, x_n]$ be the polynomial ring in $n$ variable over a field $K$. The {\em edge ideal} of $G$, denoted by $I(G)$, is 
\[
I(G)=(x_ix_j : \{x_i,x_j\} \in E(G)).
\]
The {\em cover ideal} of $G$, denoted by $J(G)$, is 
\[
J(G)=(\prod_{x_i \in T}x_i: T \text{ is a minimal vertex cover of } G).
\]
In \cite{SM}, the {\em closed neighborhood ideal} of $G$ has been introduced as 
 \[
NI(G)= (\prod_{x_j \in N[x_i]} x_j:  x_i \in V(G)).
\]
Moreover, in \cite{SM}, the dominating ideal of $G$ is defined as
\[
DI(G)=(\prod_{x_i \in S}x_i: S \text{ is a minimal dominating set of } G).
\]
It is a well-recognized fact that $J(G)$ is the Alexander dual of $I(G)$ (cf.  \cite[Lemma 9.1.4]{HH1}). It is shown in \cite[Lemma 2.2]{SM} that $DI(G)$ is the Alexander dual of $NI(G)$. 


A {\em path} in $G=(V(G), E(G))$ is a sequence of distinct vertices $x_{i_1},\ldots, x_{i_{t}}$ such that $\{x_{i_j}, x_{i_{j+1}}\}\in E(G)$ for $j=1,\ldots,t-1$.  The {\em length}  of a path is the number of edges in it. We will refer to a path of length $t-1$ as a $t$-path. If a graph $G$ itself is a path on $n$ vertices, then it is denoted by $P_n$. 
A {\em cycle}, denoted by $C_n$, is a graph with $V(C_n)=\{x_{1},\ldots, x_{n}\}$ and 
\[
E(C_n)= \{\{x_{i}, x_{i+1}\}:  i=1, \ldots, n-1\}\cup\{\{x_1,x_n\}\}.
\] 

A graph $G$ is said to be {\em chordal} if it has no induced cycles of length greater than four. Equivalently, a graph $G$ is chordal if and only if it admits a perfect elimination ordering of its vertices. A {\em perfect elimination ordering} of $G$ is an ordering $v_1, \ldots, v_n$ of $V(G)$ with the property that for each $i,j$, and $l$, if $i<j$, $i<l$ with $v_l, v_j \in N_G[v_i]$, then $v_l \in N_G[v_j]$. Note that  chordal graphs are sometimes referred to as {\em triangulated graphs}.


\subsection{Hypergraph theory} Here, we recall some notions from hypergraph theory. A finite {\it hypergraph} $\mathcal{H}$ on a vertex set $V({\mathcal{H}})=\{x_1,x_2,\ldots,x_n\}$ is a collection of edges  $E({\mathcal{H}})=\{ E_1, \ldots, E_m\}$ with $E_i \subseteq V({\mathcal{H}})$, for all $i=1, \ldots,m$.  The {\em incidence matrix} of $\mathcal{H}$ is a $n \times m$ matrix $A=(a_{ij})$ such that $a_{ij}=1$ if $x_i \in E_j$, and $a_{ij}=0$ otherwise. 
A hypergraph $\mathcal{H}$ is called {\it simple}, if $E_i \subseteq  E_j$ implies  $i = j$. Simple hypergraphs are also known as {\em clutters}. Moreover, if $|E_i|=d$, for all $i=1, \ldots, m$, then $\mathcal{H}$ is called a {\em $d$-uniform} hypergraph. A $2$-uniform hypergraph $\mathcal{H}$ is just a finite simple graph.

Let $V(\H)=\{x_1,x_2,\ldots, x_n\}$ and $R=K[x_1, \ldots, x_n]$ be the polynomial ring in $n$ variable over a field $K$. The {\it edge ideal} of $\mathcal{H}$ is given by
$$I(\mathcal{H}) = (\prod_{x_j\in E_i} x_j : E_i\in  E({\mathcal{H}})).$$


A subset  $W \subseteq V({\mathcal{H}})$  is a {\it vertex cover} or {\em transversal} of $\mathcal{H}$ if $W \cap E_i\neq \emptyset$  for all $i=1, \ldots, m$. A transversal $W$ is {\it minimal} if no proper subset of $W$ is a transversal of  $\mathcal{H}$. The family of minimal transversals of $\mathcal{H}$ constitutes a simple hypergraph on $V({\mathcal{H}})$ called the {\em transversal hypergraph} of $\mathcal{H}$, and denoted by $T_r (\mathcal{H})$. The cover ideal of the hypergraph  $\mathcal{H}$, 
denoted by $J(\mathcal{H})$, is given by 
\[
J(\mathcal{H})=(\prod_{x_i \in W}x_i: W \text{ is a minimal transversal of } \mathcal{H}).
\]
In combinatorial commutative algebra, the transversals of a hypergraph are referred to as covers, and that is why the ideal $J(\mathcal{H})$ is usually called the cover ideal of $\mathcal{H}$. Moreover, similar to the case of edge ideal of graphs, the cover ideal $J(\mathcal{H})$ is  the Alexander dual of $I(\mathcal{H})$, that is,  $J(\mathcal{H})=I(\mathcal{H})^{\vee}$, consult  \cite[Theorem 6.3.39]{V1}.

Let $G$ be a graph with $V(G)=\{x_1, \ldots, x_n\}$. The {\em neighborhood hypergraph} of $G$, denoted by $\mathcal{N}(G)$, is a hypergraph on the vertex set $V(G)$ with edges $N[x_i]$, for all $i=1, \ldots, n$. Note that the edge ideal of $\N(G)$ is the closed neighborhood ideal $NI(G)$. A reader may find a confusing that we use the word ``closed" with ideal but not with hypergraphs. It is because we want to keep the standard notations given in \cite{B} and the naming of ideal given in \cite{SM}. Following \cite[Theorem 6.3.39]{V1}, we can rephrase \cite[Lemma 2.2]{SM} as follows:

\begin{Proposition} \label{transversal}
Let $G$ be a graph. Then $DI(G)$ is the edge ideal of the transversal hypergraph of $\mathcal{N}(G)$. 
\end{Proposition}

The {\em $t$-path hypergraph} of $G$, denoted by $\mathcal{H}_t(G)$, is a hypergraph on the vertex set $V(G)$ whose edges are all distinct $t$-paths in $G$. The edge ideal of $\mathcal{H}_t(G)$ is called the {\em $t$-path ideal} of $G$ and is denoted by $I_t(G)$.  When $t=2$, then $I_2(G)$ is simply the {\em edge ideal}  $I(G)$ of $G$. If there is no $t$-path in $G$, then we set $I_t(G)=0$. 


\subsection{Simplicial complexes} 

 A {\em simplicial complex} $\Delta$ on a set of vertices $V(\Delta)=\{x_1, \ldots, x_n\}$ is a collection of subsets of $V(\Delta)$, with the property that if $F \in \Delta$,  then all subsets of $F$ are also in $\Delta$, including the empty set. An element of $\Delta$ is called a {\em face} of $\Delta$, and 
  the maximal faces of $\Delta$ with respect to inclusion are called {\em facets} of $\Delta$. The \textit{facet ideal} of $\Delta$ is given by
\[
I(\Delta)= (\prod_{x_i \in F} x_{i}: F \text{ is a facet of } \Delta).   
\]
The set of all facets of $\Delta$ can be viewed as a simple hypergraph. In this case, the definition of the edge ideal of hypergraph consisting of all the facets of $\Delta$ is  the  same as the facet ideal of $\Delta$. 


\section{Neighborhood hypergraphs and their transversal hypergraphs}\label{section2}

The edge ideals and cover ideals of bipartite graphs are known to be normally torsion-free, see \cite{SVV} and \cite{GRV}. The bipartite graphs are characterized as the graphs that avoid any cycle of odd lengths. With the help of {\it Macaulay2} \cite{GS}, one can verify that $NI(C_8)$ and $DI(C_8)$ are not normally torsion-free. This shows that one cannot directly extend the results of \cite{SVV} and \cite{GRV} in the case of closed neighborhood ideals and dominating ideals of bipartite graphs. However, if we consider the neighborhood hypergraph $\N(G)$ of a graph $G$ such that there are no hypergraphic cycles in $\N(G)$, then one obtains a similar statement as in the case of edge ideals and cover ideals of bipartite graphs. To establish this, we recall some definitions and notions from hypergraph theory. For more details about these concepts, we refer the  reader to \cite{B}.

Let $\H$ be a simple hypergraph with $n$ vertices and $m$ edges, and let $M$ be the incidence matrix of $\H$.  A hypergraph $\mathcal{H}$ is called {\em Mengerian} if for all $\cb \in  {\mathbb{Z}^n_{+}}$, we have the following equality 
\[
\min\{\ab \; \cdot \;\cb : \; \ab \in {\mathbb{Z}^n_{+}},\; M \; \cdot \;\ab \geq \1b\}=\max\{\bb  \; \cdot \; \1b:\; \bb \in  {\mathbb{Z}^m_{+}},\;  M^T \; \cdot \; \bb \leq \cb\},
\]
where $\mathbb{Z}_{+}$ denotes the nonnegative integers. 
Algebraically, it is equivalent to $I(\mathcal{H})$ being normally torsion-free; it  can be found in  \cite[Corollary 10.3.15]{HH1} and  \cite[Theorem 14.3.6]{V1}. We next  define some of the well-known classes of Mengerian hypergraphs. 

\begin{Definition}\label{hypercycle}
{\em A cycle of length $k \geq 2$ in a hypergraph $\mathcal{H}$ is a sequence 
\[\mathcal{C}: x_1, E_1, x_2, E_2, \ldots, x_k, E_k, x_1\] such that

\begin{enumerate}
\item $E_1, \ldots, E_k$ are distinct edges of $\mathcal{H}$,
\item $x_1, \ldots, x_k$ are distinct vertices of $\mathcal{H}$,
\item $x_1,x_k \in E_k$ and $x_i,x_{i+1} \in E_i$ for all $i=1, \ldots, k-1$. 
\end{enumerate}

The vertices of $\mathcal{C}$ are $x_1, \ldots, x_k$ and the edges of $\mathcal{C}$ are $E_1, \ldots, E_k$. Following \cite[Section 10.3.4]{HH1}, in the following text, a cycle $\mathcal{C}$ of length $k \geq 3$ is said to be {\em special} if all the edges of $\C$ contain  exactly two vertices of $\mathcal{C}$.  Note that if we reduce to the case of graphs,  then the special cycles of length $k \geq 3$ are just the chordless cycles. 
}
\end{Definition}
A hypergraph $\mathcal{H}$ is called {\em balanced} if every cycle of $\mathcal{H}$ of odd length $k \geq 3$ has an edge containing at least three vertices of the cycle. Moreover, $\mathcal{H}$ is called {\em totally balanced} if every cycle of $\mathcal{H}$ of length $k \geq 3$ has an edge containing at least three vertices of the cycle, see \cite[Chapter 5]{B}.  In other words, $\mathcal{H}$ is balanced if it does not have special odd cycles and $\mathcal{H}$ is totally balanced if it does not have any special cycles. It  follows from the definition that every totally balanced hypergraph is also balanced.

\begin{Remark}\label{rem:imp}{\em
A well-known result of Fulkerson, Hoffman, and Oppenheim in \cite{FHO} states that balanced hypergraphs are Mengerian. In particular, totally balanced hypergraphs are also Mengerian. Furthermore, \cite[Theorem 19 on  page 207]{B} states that transversal hypergraphs of simple balanced hypergraphs are also Mengerian.
}
\end{Remark}

It is known that the neighborhood hypergraphs of trees are totally balanced, for example see \cite[Example 3 on page 174]{B}. 
We can now combine together  Remark~\ref{rem:imp} and Proposition~\ref{transversal} to obtain the following theorem. In particular, it is 
an affirmative answer to   \cite[Question 2.8(i)]{NQBM} and   \cite[Question 2.20(i)]{NQBM}.
\begin{Theorem}\label{NI-TREES}
The closed neighborhood ideals  and dominating ideals of trees are normally torsion-free.  
\end{Theorem}


It is natural to ask that if there exists a complete characterization of graphs for which their associated neighborhood hypergraphs are balanced or totally balanced. Such graphs are characterized in \cite{Fa} and \cite{BDS}. To describe these graphs we need to recall some definitions from  \cite{Fa} and \cite{BDS}. 

\begin{Definition}{\em
Let $n \geq 3$. An {\em incomplete trampoline} is a chordal graph $G$ on $2n$ vertices whose vertex set can be partitioned into two sets, $W=\{y_1, \ldots, y_n\}$ and $U=\{u_1, \ldots, u_n\}$ such that no vertices of $W$ are adjacent to each other, and for each $i$ and $j$, $w_i$ is adjacent to $u_i$ if and only if $i=j$ or $i \equiv j+1 \;(\mathrm{mod\;}n)$. In addition, an incomplete trampoline is called {\em trampoline} if the induced subgraph on $U$ is a complete graph. If $n$ is odd, then (incomplete) trampoline is called an odd (incomplete) trampoline.}
\end{Definition}

Trampolines on $2n$ vertices are sometimes  also mentioned as {\em $n$-sun graphs}. It is shown in \cite[Theorem 2]{BDS} that for a given graph $G$, we have $\N(G)$ is balanced if and only if $G$ is chordal with no induced odd incomplete trampoline. Once again, we can deduce from Remark~\ref{rem:imp} and Proposition~\ref{transversal} the next theorem. 

\begin{Theorem}\label{trampoline}
Let $G$ be a chordal graph with no induced odd incomplete trampoline. Then the neighborhood ideal and the dominating ideal of $G$ is normally torsion-free.
\end{Theorem}

\begin{Remark}\label{nonbalanced}
{\em It is crucial to mention that there are Mengerian neighborhood hypergraphs that are not balanced. To see this,  suppose 
 that  $C_6=(V(C_6), E(C_6))$ is  the even  cycle graph  of order $6$ on the vertex set $V(C_6)=\{x_1, \ldots, x_6\}$ and the   edge set 
$E(C_6)=\{\{x_i, x_{i+1}\}~:~i=1, \ldots, 5\}\cup \{\{x_6,x_1\}\}.$  We  can view  the $C_9$ in a similar way. 
The cycle  graphs $C_6$ and $C_9$ are not chordal, and hence  \cite[Theorem 2]{BDS} tells us  $\N(C_6)$ and $\N(C_9)$ are not balanced. Indeed, the edges of $\N(C_6)$ are
\[
N[x_1]=\{x_1, x_2,x_6\}, N[x_2]=\{x_1, x_2, x_3\}, N[x_3]=\{x_2, x_3, x_4\}, 
\]
\[N[x_4]=\{x_3, x_4, x_5\},N[x_5]=\{x_4, x_5, x_6\},N[x_6]=\{x_5, x_6, x_1\}, 
\]
 and the sequence $\mathcal{C}: x_1, N[x_2], x_3, N[x_4], x_5, N[x_6], x_1$ gives a special odd cycle in the hypergraph $\N(C_6)$. One can construct a special odd cycle  in $\N(C_9)$ in a similar fashion. In Section~\ref{DIofcycles}, we will show that not only $NI(C_6)$ and $NI(C_9)$ are normally torsion-free but also  their Alexander duals $DI(C_6)$ and $DI(C_9)$ are  normally torsion-free. This highlights that the converse of Theorem~\ref{trampoline} is not true.
}

\end{Remark}

Now, we turn our attention to  the characterization of graphs for which the associated neighborhood hypergraphs are totally balanced. 
To accomplish this, one has to recall the following definitions. 

A {\em strong elimination ordering} of $G$ is an ordering $v_1, \ldots, v_n$ of $V(G)$ with the property that for each $i,j,k,$ and $l$, if $i<j$, $k<l$ with $v_k, v_l \in N_G[v_i]$ and $v_k \in N_G[v_j]$, then $v_l \in N_G[v_j]$.  Graphs that admit a strong elimination ordering are called {\em strongly chordal graphs}. The strongly chordal graphs were introduced by Farber in \cite{Fa}. They form a subclass of chordal graphs, because a strong elimination ordering is also a perfect  elimination ordering, by letting $i=k$. Among many other characterizations of strongly chordal graphs, it is shown in \cite{Fa} that a graph $G$ is strongly chordal if and only if it is chordal and does not contain any induced trampoline.
   In \cite[Theorem 3.3]{Fa}, a polynomial procedure to recognize strongly chordal graphs and to construct strong elimination orderings is given. Some of the special subclasses of strongly chordal graphs are trees and interval graphs. It is proved in \cite[Theorem 5.2]{Fa} and re-proved in \cite[Theorem 2]{BDS} that a graph $G$ is strongly chordal if and only if $\N(G)$ is totally balanced. We provide a neat  and straightforward proof of the forward implication in the following

\begin{Theorem}\label{strongly chordal}
Let $G$ be a strongly chordal graph. Then $\mathcal{N}(G)$ is totally balanced. 
\end{Theorem}
\begin{proof}
Let $G$ be a strongly chordal graph.  To simplify the notation in the following text, for any $p,q \in V(G)$ we write $p<q$ if $p$ appears before $q$ in the strong elimination ordering of $G$.
Let $$\mathcal{C}: x_1, E_1, x_2, E_2, \ldots, x_k, E_k, x_1,$$  be a cycle in $\mathcal{N}(G)$ with $k \geq 3$.  Let  $v_1, \ldots, v_k \in G$ be such that $E_i=N[v_i]$  for all $i=1, \ldots, k$. We can rearrange the edges in $\mathcal{C}$ such that $v_1<v_i$ for all $i=2, \ldots, k$, and $x_2<x_1$.  Since $x_1, x_2 \in E_1=N[v_1]$ and $x_2 \in E_2=N[v_2]$, it follows that $x_1 \in E_2$. This shows that 
$\mathcal{C}$ is not a special cycle, and so $\mathcal{N}(G)$ is totally balanced, as claimed.  
\end{proof}

Totally balanced hypergraphs have further nice algebraic implications. It follows from \cite[Theorem 2.3]{HHTZ} that a simple hypergraph $\H$ is a simplicial forest if and only if $\H$ is totally balanced. Simplicial forests were introduced by Faridi in \cite{Far1} in the language of simplicial complexes. The facet ideals of simplicial forests possess many nice properties as noted in \cite{Far1} and \cite{Far2}. For example, Faridi proved in \cite[Corollaries  5.5 and  5.6]{Far2} that if $\Delta$ is a simplicial forest,  then its facet ideal $I(\Delta)$ is sequentially Cohen-Macaulay and $I(\Delta)^{\vee}$ is componentwise linear. More details and related definitions can be found in \cite{Far1} and \cite{Far2}.

We summarize the  discussion above  together with Remark~\ref{rem:imp} and Proposition~\ref{transversal}, in the next theorem. 

\begin{Theorem}\label{stronglychordal}
Let $G$ be a strongly chordal graph. Then the following statements hold:   
\begin{itemize}
\item[(i)]  NI(G) is normally torsion-free. 
\item[(ii)]  NI(G) is sequentially Cohen-Macaulay.
\item[(iii)] DI(G) is normally torsion-free. 
\item[(iv)]  DI(G) is componentwise linear.
\end{itemize}
\end{Theorem}


\section{The $t$-path hypergraphs of path graphs}\label{tpath}

Let $P_n$ be a path on $n$ vertices. It follows from \cite[Theorem 2.7]{JT} that $\H_t(P_n)$ is a simplicial forest and using \cite[Theorem 3.2]{HHTZ} it can be seen that $\H_t(P_n)$ is totally balanced. Consequently, $I_t(P_n)$ is normally torsion-free. Furthermore,  we conclude   from Remark~\ref{rem:imp} 
  that $I_t(P_n)^{\vee}$ is normally torsion-free. We provide another proof of these facts by using algebraic tools. To do this, we first  introduce a technique to construct new normally torsion-free square-free monomial ideals based on the existing ones. This technique will be implemented in the subsequent result. 

\begin{theorem} \label{NTF1}
Let $I$ be a normally torsion-free square-free monomial ideal in a polynomial ring  $R = K[x_1, \ldots, x_n]$ and $v$ be a square-free monomial in $R$. Let  $\mathcal{G}(I)=\{u_1, \ldots, u_m\}$ such that $\mathrm{gcd}(v, u_i)$ divides  $\mathrm{gcd}(v, u_{i+1})$  for all $i=1, \ldots, m-1$.   Then the following statements hold:
\begin{itemize}
\item[(i)] $I+vR$  is normally torsion-free. 
\item[(ii)] $I+vR$  is nearly normally torsion-free. 
\item[(iii)]  $I+vR$ is normal.
\item[(iv)]  $I+vR$ has the strong persistence property. 
\item[(v)]  $I+vR$ has the persistence property. 
\item[(vi)]  $I+vR$ has the symbolic strong persistence property.
\end{itemize}
\end{theorem}

\begin{proof}
(i)  To simplify the notation, set  $J:=I+vR$.  If  $J\setminus x_t=\mathfrak{m}\setminus x_t$ for some $1\leq t \leq n$, then  we  can write  $J=x_tL +\mathfrak{m}\setminus x_t$. If $L=R$, then $J=\mathfrak{m}$, and the assertion holds. Assume that   $L\neq R$. Let  $h\in \mathcal{G}(L)$. If $x_\alpha 
 \mid h$ for some $\alpha \in \{1, \ldots, n\}\setminus \{t\}$, this gives that  $h\in  \mathfrak{m}\setminus x_t$, and hence 
 $L\subseteq \mathfrak{m}\setminus x_t$. We thus get  $J=\mathfrak{m}\setminus x_t$, and the proof is over. Thus, we assume that 
    $J\setminus x_t \neq \mathfrak{m}\setminus x_t$ for all  $t=1, \ldots,  n$.  In what follows,  we first show that $v\in \mathfrak{p}\setminus \mathfrak{p}^2$ for any $\mathfrak{p}\in \mathrm{Min}(J)$. Pick  an arbitrary element   $\mathfrak{p}\in \mathrm{Min}(J)$. Due to  $v\in J$ and $J\subseteq \mathfrak{p}$, we have  $v \in \mathfrak{p}$.
   On the contrary, assume that  $v\in \mathfrak{p}^2$. Because $v$ is  a  square-free monomial,  this yields that  
    $|\mathrm{supp}(\mathfrak{p}) \cap \mathrm{supp}(v)| \geq 2$.    If $\mathrm{gcd}(v, u_m)=1$, then  $\mathrm{gcd}(v, u_i)=1$ for all $i=1, \ldots, m-1$. In this case, the claim can be deduced directly from  Theorem   \cite[Theorem 2.5]{SN}. Let     $\mathrm{supp(gcd}(v,u_m))=\{x_{i_1}, \ldots, x_{i_k}\}$  with    $\{{i_1}, \ldots, {i_k}\}\subseteq \{1, \ldots, n\}.$  Here, one may consider the following cases: 
\vspace{1mm}

\textbf{Case 1.}  $x_{i_r} \in \mathfrak{p}$ for some $x_{i_r}\in \{x_{i_1}, \ldots, x_{i_k}\}$.  Since  $|\mathrm{supp}(\mathfrak{p}) \cap \mathrm{supp}(v)| \geq 2$, we can take   $x_\lambda\in  \mathrm{supp}(\mathfrak{p}) \cap \mathrm{supp}(v)$  with  $x_\lambda \neq x_{i_r}$. We claim  that  $x_\lambda \notin  \{x_{i_1}, \ldots, x_{i_k}\}$. Suppose, on the contrary, that  
$x_\lambda \in \{x_{i_1}, \ldots, x_{i_k}\}$.  Let    $$c:=\min\{j~:~x_{i_r}\mid u_j, \text{~where~} j=1, \ldots, m\},$$  and  
 $$d:=\min\{j~:~x_\lambda\mid u_j, \text{~where~} j=1, \ldots, m\}.$$  If   $c\leq d$ (respectively,  $c>d$), then the assumption yields that 
   $I \subseteq \mathfrak{p}\setminus \{x_\lambda\}$ (respectively, $I \subseteq \mathfrak{p}\setminus \{x_{i_r}\}$), and by virtue of 
$v \in \mathfrak{p}\setminus \{x_\lambda\} $ (respectively, $v \in \mathfrak{p}\setminus \{x_{i_r}\}$), this gives rise to  $J \subseteq \mathfrak{p}\setminus \{x_\lambda\}$  (respectively,  $J \subseteq \mathfrak{p}\setminus \{x_{i_r}\}$), which contradicts the minimality of 
 $\mathfrak{p}$.  Consequently, we conclude that $x_\lambda \notin  \{x_{i_1}, \ldots, x_{i_k}\}$, and thus   $x_\lambda \notin \mathrm{supp}(I)$. 
    In the light of  $v \in \mathfrak{p}\setminus \{x_\lambda\}$, this leads to   $J \subseteq \mathfrak{p}\setminus \{x_\lambda\}$, a contradiction to the minimality of $\mathfrak{p}$.
\vspace{1mm}

\textbf{Case 2.}    $x_{i_r} \notin  \mathfrak{p}$ for any  $x_{i_r}\in \{x_{i_1}, \ldots, x_{i_k}\}$.   Select two distinct variables  $x_\alpha, x_\beta \in \mathrm{supp}(\mathfrak{p}) \cap \mathrm{supp}(v)$.  Accordingly, we have $x_\alpha, x_\beta \notin \mathrm{supp}(I)$. We thus  get 
 $I \subseteq \mathfrak{p}\setminus \{x_\alpha\} $. On account of  $v \in \mathfrak{p}\setminus \{x_\alpha\} $,  one can  conclude that 
 $J \subseteq  \mathfrak{p}\setminus \{x_\alpha\} $. This contradicts the  minimality of $\mathfrak{p}$.

Therefore,  we deduce that $v\notin \mathfrak{p}^2$.  In addition, note that    $J\setminus x_t=I\setminus x_t$ for any $x_t \in \mathrm{supp}(v)$.  It follows also from  \cite[Theorem 3.21]{SN}  that   $I\setminus x_t$ is normally torsion-free  for any $x_t \in \mathrm{supp}(v)$, and so   $J\setminus x_t$ is normally torsion-free  for any $x_t \in \mathrm{supp}(v)$.        Fix $s\geq 1$ and  $x_t \in \mathrm{supp}(v)$.    Suppose, on the contrary, that 
$\mathfrak{m}\setminus x_t \in \mathrm{Ass}(R/(J\setminus x_t)^s)$. In view of 
$\mathrm{Ass}(R/(J\setminus x_t)^s)=\mathrm{Min}(J\setminus x_t)$, we derive that    $\mathfrak{m}\setminus x_t \in\mathrm{Min}(J\setminus x_t)$, and thus  $J\setminus x_t=\mathfrak{m}\setminus x_t,$ which is a contradiction. We therefore have  $\mathfrak{m}\setminus x_t \notin \mathrm{Ass}(R/(J\setminus x_t)^s)$ for all $s$ and $x_t \in \mathrm{supp}(v)$.   It follows now from \cite[Theorem 3.7]{SNQ} that $J$ is  normally torsion-free, and the proof  is complete.  \par   
 (ii)-(vi) can be shown similar to the proof of  \cite[Corollary 4.10]{NQKR}.   
\end{proof}


As an application of Theorem \ref{NTF1}, we give the following result. 

\begin{proposition} \label{Path.NTF}
The $t$-path ideals of path graphs are normally torsion-free, for all $t\geq 0$.
\end{proposition}

\begin{proof}
Let $P_n$ be a path graph with the vertex set  $V(P_n)=\{x_1, \ldots, x_n\}$ and  the edge set  
$E(P_n)=\{\{x_i, x_{i+1}\}~:~ i=1, \ldots, n-1\}.$  Then the  $t$-path ideal of $P_n$ is given by 
$I_t(P_n)=(x_{i}x_{i+1}\cdots x_{i+t-1}~:~i=1, \ldots, n-t+1).$ 
The proof is  by induction on $n$. If $n=t$, then $I_t(P_n)=(x_1x_2 \cdots x_{t})$, and there is nothing to show. Let $n>t$ and that the claim has been proven  for any value less than $n$.   Now, put  $I:=(x_{i}x_{i+1}\cdots x_{i+t-1}~:~i=1, \ldots, n-t)$,  $v:=x_{n-t+1}x_{n-t+2}\cdots x_{n-1}x_n$, and  
$u_i:=x_{i}x_{i+1}\cdots x_{i+t-1}$ for all $i=1, \ldots, n-t$. It is easily seen that    $\mathrm{gcd}(v,u_i)=1$ for all $i=1, \ldots, n-2t+1$, and 
$\gcd(v, u_i)=x_{n-t+1}x_{n-t+2}\cdots x_{t+i-1}$ for all $i=n-2t+2, \ldots, n-t$.  In particular, this implies that   $\mathrm{gcd}(v, u_i)$ divides  $\mathrm{gcd}(v, u_{i+1})$  for all $i=1, \ldots, n-t-1$.   It  follows also from the induction hypothesis that $I$ is normally torsion-free.  Thanks to     $I_t(P_n)=I+vR$, where $R=K[x_1, \ldots, x_n]$, we deduce  immediately  the assertion from  Theorem \ref{NTF1}(i). 
 \end{proof}


In what follows, our aim is to  study the  normally torsion-freeness of the Alexander duals of $3$-path ideals  of path graphs. 
For this, we will need the following proposition.


\begin{proposition}\label{Intersection}
Suppose that  $I$ and  $J$  are   two  normally torsion-free square-free monomial ideals in a polynomial ring  $R=K[x_1, \ldots, x_n]$ over a field $K$ such that  $\mathrm{supp}(I) \cap \mathrm{supp}(J)=\emptyset$.  Then $I\cap J=IJ$ is normally torsion-free.  
\end{proposition}
\begin{proof}
Due to    $\mathrm{supp}(I) \cap \mathrm{supp}(J)=\emptyset$, we  deduce from \cite[Proposition 1.2.1]{HH1}  that  $IJ=I \cap J$. 
In particular,  we have 
$\mathrm{Min}(I\cap J)=\mathrm{Min}(I) \cup  \mathrm{Min}(J)$. On the other hand, we know from \cite[Theorem 1.4.6]{HH1} that  a square-free monomial ideal $L$ is normally torsion-free if and only if $L^k=L^{(k)}$ for all $k\geq 1$, where $L^{(k)}$ denotes the $k$-th symbolic power of $L$. 
Hence, it is enough for us to show that $(I\cap J)^k=(I\cap J)^{(k)}$ for all $k\geq 1$. To  achieve this, fix $k\geq 1$. 
Since $I$ and  $J$  are    normally torsion-free,  one can conclude that $I^k=I^{(k)}$ and $J^k=J^{(k)}$. Thus, we get the following equalities  
\begin{align*}
(I \cap J)^{(k)}&=\bigcap_{\mathfrak{p}\in \mathrm{Min}(I\cap J)}((I \cap J)^kR_\mathfrak{p}\cap R)\\
&=\bigcap_{\mathfrak{p}\in \mathrm{Min}(I)}(I^kR_\mathfrak{p}\cap R) \cap \bigcap_{\mathfrak{p}\in \mathrm{Min}(J)}(J^kR_\mathfrak{p}\cap R)\\
&=I^{(k)}\cap  J^{(k)} \\
& =I^k \cap J^k\\
&=(I\cap J)^k. 
\end{align*}
This shows that $I\cap J=IJ$ is normally torsion-free, as claimed. 
\end{proof}


To prove the next propostion, we will use  the following auxiliary result. 

\begin{lemma} (\cite[Lemma 2.17]{NQBM}) \label{NQBM}
Let $I$ be a normally torsion-free square-free  monomial ideal in a polynomial ring $R=K[x_{1},\ldots ,x_{n}]$ with $\mathcal{G}(I) \subset R$. Then the ideal $$L:=IS\cap (x_{n}, x_{n+1}, x_{n+2}, \ldots, x_m)\subset  S=R[x_{n+1}, x_{n+2}, \ldots, x_m],$$  is normally torsion-free.
\end{lemma}


\begin{proposition} \label{NTF-Alexander Dual-Path}
The Alexander duals of $3$-path ideals of path graphs are normally torsion-free square-free monomial ideals.
\end{proposition}

\begin{proof}
Let $P_n$ be a path graph with the vertex set $V(P_n)=\{x_1, \ldots, x_n\}$ and the edge  set $E(P_n)=\{\{x_i, x_{i+1}\}~:~ i=1, \ldots, n-1\}$.
 Then  the Alexander dual of the  $3$-path ideal of the path graph $P_n$, denoted by $I_3(P_n)^{\vee}$,   is given by 
$I_3(P_n)^{\vee}=\bigcap_{i=1}^{n-2}(x_i, x_{i+1}, x_{i+2}).$ 
We show the claim by induction on $n$. Since every prime monomial ideal is normally torsion-free, this implies that the assertion is true for the case in which $n=3$. Let $n>3$ and that  the claim has been shown for any value less than $n$. Our strategy is to use  \cite[Theorem 3.7]{SNQ}. For this purpose, set 
 $\mathfrak{p}_i:=(x_i, x_{i+1},  x_{i+2})$ for all $i=1, \ldots, n-2$,  and $$v:=\prod_{k=0}^{[\frac{n-1}{3}]}x_{1+3k}.$$ 
It  is easy  to check   that $|\mathrm{supp}(v) \cap \mathrm{supp}(\mathfrak{p}_i)|=1$ for all $i=1, \ldots, n-2$. This yields that $v\in \mathfrak{p}_i \setminus \mathfrak{p}_i^2$ for all $i=1, \ldots, n-2$. To complete the argument, one has to verify that 
$\mathfrak{m}\setminus x_i \notin \mathrm{Ass}(R/(I_3(P_n)^{\vee}\setminus x_i)^s)$ for all $s$ and $x_i \in \mathrm{supp}(v)$, where $\mathfrak{m}=(x_1, \ldots, x_n)$. Fix $s\geq 1$ and $x_i \in \mathrm{supp}(v)$.  One can  check  that $I_3(P_n)^{\vee}\setminus x_i$, after deleting the redundant components,  can be viewed  as $J_1 \cap J_2 \cap \Gamma$, where $J_1=\bigcap_{r=1}^{\ell_1}\mathfrak{q}_r$ (respectively,  $J_2=\bigcap_{t=1}^{{\ell_2}}\mathfrak{q}'_t$) is  the Alexander dual  of a $3$-path ideal of a  path graph such as $P_1$ (respectively, $P_2$)  with $V(P_1)<n$ (respectively, $V(P_2)<n$), $V(P_1) \cap V(P_2)=\emptyset$ and $\Gamma=\bigcap_{c=1}^{\lambda}(x_{\alpha_c}, x_{\beta_c})$ such that  for any $1\leq r \leq \ell_1$,    $1\leq t \leq \ell_2$,  and 
$1\leq c\neq d \leq \lambda$, we have $|\mathrm{supp}(\mathfrak{q}_r) \cap  \mathrm{supp}(x_{\alpha_c}, x_{\beta_c})|\leq 1$, 
$|\mathrm{supp}(\mathfrak{q}'_t) \cap  \mathrm{supp}(x_{\alpha_c}, x_{\beta_c})|\leq 1$, 
 and 
$|\mathrm{supp}(x_{\alpha_c}, x_{\beta_c})  \cap \mathrm{supp}(x_{\alpha_d}, x_{\beta_d})|= 1$.  Furthermore, the inductive hypothesis gives  that $J_1$ and $J_2$ are  normally torsion-free, and Proposition \ref{Intersection} implies that $J_1\cap J_2$ is normally torsion-free. It follows now by  repeated use of  Lemma \ref{NQBM} that $J_1 \cap J_2 \cap \Gamma$  is normally torsion-free, and so  $I_3(P_n)^{\vee}\setminus x_i$ is normally torsion-free. This means that $\mathfrak{m}\setminus x_i \notin \mathrm{Ass}(R/(I_3(P_n)^{\vee}\setminus x_i)^s)$, as required. We conclude from \cite[Theorem 3.7]{SNQ} that $I_3(P_n)^{\vee}$ is normally torsion-free. 
This completes the  inductive step, and  therefore  the claim has been shown by induction.  
\end{proof}


\section{Stable sets of associated primes of dominating ideals of cycles }\label{DIofcycles}

In this section, we will mainly discuss the normally torsion-freeness  of $DI(C_n)$ and the stable sets of associated primes of $DI(C_n)$. In \cite[Theorem 3.9]{AB},  Alilooee  and  Banerjee  showed that the $3$-path ideal of a graph $G$ is normally torsion-free if and only if $G$  is  a path graph $P_k$ for some $k\geq 3$ or $G$ is a cycle $C_{3k}$ when $k=1, 2, 3$. Observe that the $3$-path ideal of a cycle $C_n$ with $n \geq 3$ coincides with $NI(C_n)$. With this observation, we can rephrase \cite[Theorem 3.9]{AB} in the following way.

\begin{proposition}\label{NI-CYCLE}
Let $C_n$ be a cycle graph with $n\geq 3$. Then  $NI(C_{n})$ is normally torsion-free if and only if $n\in \{3,6,9\}$. 
  \end{proposition}

As the first main result of this section, in Theorem~\ref{symbolic-cycles}, we will prove that the  statement above also holds for $DI(C_n)$. To  show   this, we first consider the case of $DI(C_6)$ and $DI(C_9)$ and verify  that they are normally torsion-free. 


\begin{lemma}\label{DI(C6+C9)}
The dominating ideals of $C_6$ and $C_9$ are normally torsion-free. 
\end{lemma}
\begin{proof}
(i)  Suppose that  $E(C_6)=\{\{x_i, x_{i+1}\} : i=1, \ldots, 5\}\cup \{\{x_6,x_1\}\}.$ Then the dominating ideal of $C_6$ is given by 
\begin{align*}
DI(C_6)=& (x_1,x_2,x_3) \cap (x_2,x_3,x_4) \cap (x_3,x_4,x_5) \cap (x_4,x_5,x_6) \cap (x_5,x_6,x_1) \\ 
& \cap (x_6,x_1,x_2)\\
=&(x_3x_6,x_2x_5,x_1x_4,x_2x_4x_6,x_1x_3x_5).
\end{align*}
We want to demonstrate that   $DI(C_6)$ is normally torsion-free by using  \cite[Theorem 3.7]{SNQ}.  For this purpose, set $v:=x_3x_6$ and $L:=DI(C_6)$. 
One can easily see that $v\in \mathfrak{p}\setminus \mathfrak{p}^2$ for any $\mathfrak{p}\in \mathrm{Min}(L)$. Here, we must show 
that  $\mathfrak{m}\setminus x_i \notin \mathrm{Ass}(R/(L\setminus x_i)^s)$ for all $s$ and $x_i \in \{x_3, x_6\}.$ Fix $s\geq 1$. First, we establish  $\mathfrak{m}\setminus x_3 \notin \mathrm{Ass}(R/(L\setminus x_3)^s)$, where $\mathfrak{m}=(x_1, \ldots, x_6)$.  According to the  definition of deletion, we obtain  
 \begin{align*}
L\setminus x_3=& (x_1,x_2) \cap (x_2,x_4) \cap (x_4,x_5)  \cap (x_5,x_6,x_1)\\
=&(x_2x_5, x_1x_4, x_2x_4x_6).
\end{align*}
 It follows from \cite[Theorem 2.5 and Lemma 3.12]{SN} that $x_4(x_1, x_2x_6)$  is  normally torsion-free. Because  $\mathrm{gcd}(x_2x_5, x_1x_4)=1$ and $\mathrm{gcd}(x_2x_5, x_2x_4x_6)=x_2$, one can derive from Theorem  \ref{NTF1} that $L\setminus x_3$ is normally torsion-free, and so $\mathfrak{m}\setminus x_3=(x_1,x_2,x_4,x_5,x_6) \notin \mathrm{Ass}(R/(L\setminus x_3)^s)$.  Now, we show that 
 $\mathfrak{m}\setminus x_6 \notin \mathrm{Ass}(R/(L\setminus x_3)^s)$. A similar computation gives that 
 \begin{align*}
L\setminus x_6=& (x_1,x_2) \cap (x_2,x_3, x_4) \cap (x_4,x_5)  \cap (x_5,x_1)\\
=&(x_2x_5, x_1x_4, x_1x_3x_5).
\end{align*}
On account on \cite[Theorem 2.5 and Lemma 3.12]{SN}, we get  $x_1(x_4, x_3x_5)$  is  normally torsion-free. As  $\mathrm{gcd}(x_2x_5, x_1x_4)=1$ and $\mathrm{gcd}(x_2x_5, x_1x_3x_5)=x_5$,  Theorem  \ref{NTF1} implies that $L\setminus x_6$ is normally torsion-free, and hence  $\mathfrak{m}\setminus x_6=(x_1,x_2, x_3, x_4,x_5) \notin \mathrm{Ass}(R/(L\setminus x_6)^s)$. It follows now from \cite[Theorem 3.7]{SNQ}  that $DI(C_6)$ is normally torsion-free, as claimed. 

(ii)  Let  $E(C_9)=\{\{x_i, x_{i+1}\} : i=1, \ldots, 8\}\cup \{\{x_9,x_1\}\}.$ Then the dominating ideal of $C_9$ is given by 
\begin{align*}
DI(C_9)=& (x_1,x_2,x_3) \cap (x_2,x_3,x_4) \cap (x_3,x_4,x_5) \cap (x_4,x_5,x_6) \cap (x_5,x_6,x_7) \\ 
& \cap (x_6,x_7,x_8)\cap (x_7,x_8,x_9)\cap (x_8,x_9,x_1) \cap (x_9,x_1,x_2).
\end{align*}
Our strategy is to use \cite[Theorem 3.7]{SNQ}. To do this, put $v:=x_3x_6x_9$ and $L:=DI(C_9)$. It is routine to investigate that  $v\in \mathfrak{p}\setminus \mathfrak{p}^2$ for any $\mathfrak{p}\in \mathrm{Min}(L)$.  
  To complete the argument, one has to verify that  $\mathfrak{m}\setminus x_i \notin \mathrm{Ass}(R/(L\setminus x_i)^s)$ for all $s$ and $x_i \in \{x_3, x_6, x_9\}.$ Fix $s\geq 1$. We first  prove that $\mathfrak{m}\setminus x_3 \notin \mathrm{Ass}(R/(L\setminus x_3)^s)$, where $\mathfrak{m}=(x_1, \ldots, x_9)$.  It follows from the definition of deletion  that 
 \begin{align*}
L\setminus x_3=& (x_1,x_2) \cap (x_2,x_4) \cap (x_4,x_5)  \cap (x_5,x_6,x_7)  \cap (x_6,x_7,x_8)\cap (x_7,x_8,x_9)\\
& \cap  (x_8,x_9,x_1). 
\end{align*}
Set $\alpha:=x_1x_4x_7$ and $A:=L \setminus x_3$. It is easy to see that $\alpha \in \mathfrak{p}\setminus \mathfrak{p}^2$ for any $\mathfrak{p}\in \mathrm{Min}(A)$. We show that  $Q\setminus x_i \notin \mathrm{Ass}(R/(A\setminus x_i)^s)$ for all $s$ and $x_i \in \{x_1, x_4, x_7\}$, where $Q=\mathfrak{m}\setminus x_3$.  Fix $s\geq 1$. We first prove that   $Q\setminus x_1 \notin \mathrm{Ass}(R/(A\setminus x_1)^s)$.   Direct computation gives that 
$$A\setminus x_1= (x_6,x_7,x_8)\cap   (x_5,x_6,x_7)  \cap  (x_8,x_9) \cap  (x_4,x_5)  \cap (x_2).$$ 

It follows from Lemma \ref{NTF-Alexander Dual-Path} that $(x_6,x_7,x_8)\cap   (x_5,x_6,x_7)$ is normally torsion-free, and by virtue of Lemma \ref{NQBM}, we get $(x_6,x_7,x_8)\cap   (x_5,x_6,x_7) \cap  (x_8,x_9) \cap  (x_4,x_5)$ is normally torsion-free. According to Proposition \ref{Intersection}, one has 
$A\setminus x_1$ is normally torsion-free, and so  $Q\setminus x_1 \notin \mathrm{Ass}(R/(A\setminus x_1)^s)$. By a similar argument, one can show that 
$$A\setminus x_4=(x_2) \cap (x_5)  \cap  (x_6,x_7,x_8)\cap (x_7,x_8,x_9)   \cap  (x_8,x_9,x_1),$$
and
$$A\setminus x_7=(x_1,x_2) \cap (x_2,x_4) \cap (x_4,x_5)  \cap (x_5,x_6)  \cap (x_6,x_8)\cap (x_8,x_9),$$
are normally torsion-free. Hence,  we get $Q\setminus x_4 \notin \mathrm{Ass}(R/(A\setminus x_4)^s)$ and  $Q\setminus x_7 \notin \mathrm{Ass}(R/(A\setminus x_7)^s)$. We deduce from \cite[Theorem 3.7]{SNQ} that $A=L \setminus x_3$ is normally torsion-free, and so 
$\mathfrak{m}\setminus x_3 \notin \mathrm{Ass}(R/(L\setminus x_3)^s)$. In what follows, we give the sketch of the proof of normally torsion-freeness of 
$B:=L \setminus x_6$  and $C:=L \setminus x_9$. We can rapidly conclude   the following 
 \begin{align*}
L\setminus x_6=& (x_4,x_5) \cap (x_5,x_7) \cap (x_7,x_8)  \cap (x_1, x_2, x_3) \cap (x_2, x_3, x_4) \cap (x_1, x_8, x_9)\\
& \cap  (x_1, x_2, x_9), 
\end{align*}
and 
 \begin{align*}
L\setminus x_9=& (x_7,x_8) \cap (x_8,x_1) \cap (x_1,x_2)  \cap  (x_2,x_3,x_4) \cap (x_3,x_4,x_5) \cap (x_4,x_5,x_6)\\
& \cap (x_5,x_6,x_7).
\end{align*}
It is routine to check that $\alpha=x_1x_4x_7 \in \mathfrak{p}\setminus \mathfrak{p}^2$ for any $\mathfrak{p}\in \mathrm{Min}(B)$ (respectively,  
for any $\mathfrak{p}\in \mathrm{Min}(C)$). Furthermore, we have the following 
$$B\setminus x_1= (x_4,x_5) \cap (x_5,x_7) \cap (x_7,x_8)  \cap (x_8, x_9) \cap (x_9,  x_2)  \cap   (x_2, x_3), $$
$$B\setminus x_4= (x_5) \cap  (x_7, x_8)  \cap (x_2, x_3) \cap (x_1, x_8, x_9) \cap (x_1, x_2, x_9), $$
$$B\setminus x_7=  (x_5) \cap (x_8)  \cap (x_1, x_2, x_3) \cap (x_2, x_3, x_4) \cap (x_1, x_2, x_9), $$
$$C\setminus x_1=(x_8) \cap (x_2) \cap  (x_3,x_4,x_5) \cap (x_4,x_5,x_6) \cap (x_5,x_6,x_7), $$
$$C\setminus x_4=(x_7,x_8) \cap (x_8,x_1) \cap (x_1,x_2)  \cap  (x_2,x_3) \cap (x_3,x_5) \cap (x_5,x_6),$$
and 
$$C\setminus x_7=(x_8) \cap (x_1,x_2) \cap (x_5, x_6) \cap (x_2,x_3,x_4) \cap (x_3,x_4,x_5).$$
 By using  Lemma \ref{NTF-Alexander Dual-Path}, \cite[Theorem 2.5]{SN}, and Lemma \ref{NQBM}, we can conclude that 
 $B\setminus x_1$, $B\setminus x_4$, and  $B\setminus x_7$ (respectively,  $C\setminus x_1$, $C\setminus x_4$, and  $C\setminus x_7$)  are normally torsion-free.    This gives rise to  $\mathfrak{m}\setminus x_6 \notin \mathrm{Ass}(R/(I\setminus x_6)^s)$ (respectively,  $\mathfrak{m}\setminus x_9 \notin \mathrm{Ass}(R/(I\setminus x_9)^s)$. Finally,  \cite[Theorem 3.7]{SNQ} gives that  $DI(C_9)$ is normally torsion-free. 
\end{proof}


We are now ready to establish  the first main result of this section in the following theorem.   

\begin{theorem}\label{symbolic-cycles}
Let $C_n=(V(C_n), E(C_n))$ be a cycle graph of order $n\geq 3$.  Then $DI(C_n)^t=DI(C_n)^{(t)}$ for all $t \geq 1$ if and only if $n\in\{3,6,9\}$. In particular, $DI(C_n)$ is normally torsion-free if and only if $n\in\{3,6,9\}$. 
\end{theorem}

\begin{proof} 
We first assume that $n\in\{3,6,9\}$. 
Let $V(C_n)=\{x_1, \ldots, x_n\}$ and $E(C_n)=\{\{x_i, x_{i+1}\} \; : i=1, \ldots, n-1\}\cup\{\{x_n,x_1\}\}$. 
It is easy to see  that  $DI(C_3)=(x_1, x_2,x_3)$ is normally torsion-free. According to  Lemma~\ref{DI(C6+C9)}, it follows that  $DI(C_6)$ and $DI(C_9)$ are also normally torsion-free. To complete the proof, we need to show that  $DI(C_n)^t \neq DI(C_n)^{(t)}$ for some $t\geq 1$ when $n\notin \{3,6,9\}$.  To do this, we divide the proof into three cases. We will repeatedly use the following facts in our proof. 
\begin{itemize}
\item[(i)] The domination number of  $C_{3k}$ is $k$, and  the domination number of $C_{3k+1}$ and $C_{3k+2}$ is $k+1$, see  \cite{HHRC} for more information. 
 \item[(ii)]  The associated primes of $DI(C_n)$ correspond to the closed neighborhoods of the vertices of $C_n$.
\end{itemize}
Now,  we proceed with the proof. 
\vspace{1mm}

\textbf{Case 1.} Let $n=3k$  for some $k \geq 4$. It follows from \cite[Lemma 4]{HHRC} that there are exactly three minimal dominating sets of $C_n$ of cardinality $k$, given by
\[
A=\{x_1, x_4, x_7, \ldots, x_{3k-2}\},  B=\{x_2, x_5, x_8, \ldots, x_{3k-1}\},\; \text{and}
\]
\[  C=\{x_3, x_6, x_9, \ldots, x_{3k}\}. 
\]
Our aim is to  construct an element $f$ with $\deg f= 2k+1$ such that none of  $A$ or $B$  or $C$ is in $\mathrm{supp}(f)$. To do this,  
set   $h:= x_4x_8\prod_{j=4}^kx_{3j}$ and $g:=\prod_{i=1}^{3k} x_i$.   
We claim that $f=g/h \notin  DI(C_{3k})^2$ but $f \in  DI(C_{3k})^{(2)}$. It is routine to check  that $\deg h = k-1$, and hence $\deg f= 3k-k+1= 2k+1$. In addition, from the constructions of $A$, $B$, and $C$, we know that  the only elements of degree $k$  in $\mathcal{G}(DI(C_{3k}))$ are 
\[
\{\prod_{x_i\in A}x_i, \prod_{x_i\in B}x_i, \prod_{x_i\in C}x_i\}.
\]
Since $x_4, x_8, x_{12} \notin \mathrm{supp}(f)$, we deduce that   $\prod_{x_i\in A}x_i \nmid f,$ $\prod_{x_i\in B}x_i\nmid f,$ and 
$\prod_{x_i\in C}x_i\nmid f$. Therefore, none of the elements of degree $k$ in  $DI(C_{3k})$ divides $f$, and consequently $f \notin DI(C_{3k})^2$.  Here, recall that 
$$DI(C_{3k})^{(2)}=\bigcap_{\mathfrak{p}\in \mathrm{Ass}(DI(C_{3k}))}\mathfrak{p}^2.$$
To see  $f \in  DI(C_{3k})^{(2)}$, note that the closed neighborhood set of each vertex of $C_{3k}$ is covered twice by $\mathrm{supp}(f)$.  Indeed, for each $x_i \in C_{3k}$, we have either $x_{i}x_{i+1} \in \mathrm{supp}(f)$, or $x_{i-1}x_{i} \in \mathrm{supp}(f)$, or $x_{i-1}x_{i+1} \in \mathrm{supp}(f)$. 
\vspace{1mm}

\textbf{Case 2.}  It follows from  $x_1 x_2 x_3\in DI(C_4)^{(2)}\setminus DI(C_4)^{2}$  that $DI(C_4)^2 \neq DI(C_4)^{(2)}$, and so the assertion is true  for $C_4$. Let $n=3k+1$ for some $k\geq 2$. Let  $f':= x_{3k+1} \prod_{i=0}^{k-1}x_{3i+1}x_{3i+2}.$
Then $\deg f'=2k+1$. By virtue of  the minimal degree of any  generator in $DI(C_{3k+1})^2$ is $2k+2$, this  shows that $f' \notin DI(C_{3k+1})^2$. Moreover, the closed neighborhood set of each vertex of $C_{3k+1}$ is covered twice by $ \mathrm{supp}(f')$, which yields  that $f' \in DI(C_{3k+1})^{(2)}$.
\vspace{1mm}

\textbf{Case  3.}  Let $n=3k+2$ for some $k\geq 1$ and  put  $f'':=\prod_{i=1}^{3k+2} x_i$. 
Then $\deg f''=3k+2$. In view of  the minimal degree of any  generator in $DI(C_{3k+2})^3$ is $3k+3$, we get  $f'' \notin DI(C_{3k+2})^3$. Furthermore, the closed neighborhood set  of each vertex of $C_{3k+2}$ is covered three times  by $\supp(f'')$, which gives  that $f'' \in DI(C_{3k+2})^{(3)}$.

Accordingly, the discussion above shows  that if $n\notin \{3,6,9\}$, then   $DI(C_n)^t \neq DI(C_n)^{(t)}$ for some $t\geq 1$. This finishes the  proof. 
\end{proof}


In general, it is well-known that the cover ideals of even cycle graphs are always normally torsion-free (cf.  \cite[Corollary 2.6]{GRV}). It has already been shown  in \cite[Theorem 4.5]{NQKR} that the cover ideals of odd cycle graphs are nearly normally torsion-free, see \cite[Definition 1.1]{NQKR}.  Furthermore, the associated primes of powers of cover ideals of odd cycle graphs have already   been  studied in \cite{NKA}, refer to the  proposition below.


\begin{proposition}(\cite[Proposition 3.6]{NKA})
 Suppose that $C_{2n+1}$ is  a cycle graph on the  vertex set $[2n+1]$, $R=K[x_1, \ldots, x_{2n+1}]$ is a  polynomial ring over a field $K$,
 and $\mathfrak{m}$ is the unique homogeneous maximal ideal  of $R$. Then  $$\mathrm{Ass}_R(R/(J(C_{2n+1}))^s)= \mathrm{Ass}_R(R/J(C_{2n+1}))\cup \{\mathfrak{m}\},$$  for all $s\geq 2$. In particular, 
 $$\mathrm{Ass}^\infty(J(C_{2n+1}))=\{(x_i, x_{i+1})~: ~ i=1, \ldots 2n\}\cup\{(x_{2n+1}, x_1)\}\cup \{\mathfrak{m}\}.$$
  \end{proposition}


In what follows, our aim is to establish that a similar result as above holds for  the  associated primes of powers of dominating ideals of cycle graphs of any order. 
 Before stating the next theorem, we  recall the definition of the  {\it  monomial localization}   of a monomial ideal with respect to a monomial prime ideal.
     Let $I$ be a monomial ideal in a polynomial ring $R=K[x_1, \ldots, x_n]$ over a field $K$. In addition, we  denote by $V^*(I)$ the set of monomial prime ideals containing $I$. Let $\mathfrak{p}=(x_{i_1}, \ldots, x_{i_r})$ be a monomial prime ideal. The {\it monomial localization} of $I$ with respect to $\mathfrak{p}$, denoted by $I(\mathfrak{p})$, is the ideal in the polynomial ring $R(\mathfrak{p})=K[x_{i_1}, \ldots, x_{i_r}]$ 
 which is obtained from $I$ by using  the $K$-algebra homomorphism $R\rightarrow R(\mathfrak{p})$  with $x_j\mapsto 1$ 
 for all $x_j\notin \{x_{i_1}, \ldots, x_{i_r}\}$.  It is well-known that 
 $\mathfrak{p}\in \mathrm{Ass}_R(R/I)$ if and only if $\mathfrak{m}_\mathfrak{p}\in  \mathrm{Ass}(I(\mathfrak{p}))$, 
 where $\mathfrak{m}_\mathfrak{p}$ is the graded maximal ideal of $R(\mathfrak{p})$. For  ease of reference, we recall the following theorem   that will be used to prove Theorem~\ref{NNTF-DI(CN)}.


 \begin{theorem} (\cite[Theorem 3.1]{NQBM})  \label{I+xH} 
Let $I$ and $H$  be two  normal square-free monomial ideals in a  polynomial ring $R=K[x_1, \ldots, x_n]$ 
such that $I+H$ is normal.
  Let  $x_{c} \in \{x_1, \ldots, x_n\}$ be  a variable with  
 $\gcd (v,x_c)=1$ for all $v\in \mathcal{G}(I)\cup  \mathcal{G}(H)$. Then $L:=I+x_cH$ is normal. 
\end{theorem}


We are now  in the position to  state our next result related to the associated primes of powers of dominating ideals  of  cycle graphs.

 \begin{theorem}\label{NNTF-DI(CN)}
 Let $C_n$ be a cycle graph of order $n$ with  $V(C_n)=\{x_1, \ldots, x_n\}$ and $E(C_n)=\{\{x_i, x_{i+1}\}~:~ i=1, \ldots, n-1\} \cup \{\{x_n, x_1\}\}$. 
 Let $DI(C_n)$ be the  dominating  ideal of $C_n$ in the polynomial ring $R=K[x_1, \ldots,x_n]$.   Then the following statements hold.
 \begin{itemize}
 \item[(i)] The monomial ideal $DI(C_n)$ satisfies the persistence property.
 \item[(ii)] For all $s\geq 1$, we have 
 $$\mathrm{Ass}(DI(C_n)^s) \subseteq \{(x_i, x_{i+1}, x_{i+2})~:~i=1, \ldots, n\} \cup \{\mathfrak{m}\},$$ 
  where $\mathfrak{m}=(x_1, \ldots, x_n)$ and   $x_{n+1}$ (respectively, $x_{n+2}$) represents $x_1$ (respectively, $x_2$).  
  \item[(iii)] $\mathrm{Ass}^{\infty }(DI(C_n))\subseteq \{(x_i, x_{i+1}, x_{i+2})~:~i=1, \ldots, n\}  \cup \{\mathfrak{m}\}.$
 \end{itemize}
  \end{theorem}
 
 \begin{proof}
For convenience of notation,  put $L:=DI(C_{n})$.  Since all claims are true for 
$DI(C_3)=(x_1,x_2,x_3)$, we assume that $n\geq 4$. 
We first show that   $\mathrm{Ass}(L^s) \subseteq \mathrm{Min}(L) \cup \{\mathfrak{m}\}$ for all $s\geq 1$. To  see this, fix $s\geq 1$, and pick an arbitrary element $\mathfrak{p}\in \mathrm{Ass}(L^s)$. If $\mathfrak{p}=\mathfrak{m}$, then there is nothing to show. We thus assume that $\mathfrak{p}\neq \mathfrak{m}$. 
Without loss of generality, one may  assume that $\mathfrak{p}=(x_i~:~ i=1,  \ldots, r)$,  where $3\leq  r<n$. 
Then we get  $L(\mathfrak{p})=\cap_{i=1}^{\ell}I^{\vee}_3(P_i)$, where $P_1, \ldots, P_\ell$ are some disjoint paths in $C_{n}$.   By virtue of 
Lemma \ref{NTF-Alexander Dual-Path},   one can immediately conclude that  $I^{\vee}_3(P_i)$ is normally torsion-free for  all $i=1, \ldots, \ell$. In  the  light  of  Proposition \ref{Intersection}, we get  $L(\mathfrak{p})$ is normally torsion-free, and hence $\mathfrak{p}\in \mathrm{Min}(L(\mathfrak{p}))$.
On the other hand, based on \cite[Lemma 4.6]{RNA}(viii), we know that 
$\mathfrak{p}\in \mathrm{Ass}(L^s)$  if and only if  $\mathfrak{p}\in \mathrm{Ass}(R(\mathfrak{p})/(L(\mathfrak{p})^s).$
Consequently, one obtains    $\mathfrak{p}\in \mathrm{Ass}(R(\mathfrak{p})/(L(\mathfrak{p})^s)$.
 We therefore get  $\mathfrak{p}=(x_{\alpha}, x_{\alpha+1}, x_{\alpha+2})$ for some $1\leq \alpha \leq r-2$, and thus $\mathfrak{p}\in \mathrm{Min}(L)$.   Accordingly, we derive that $\mathrm{Ass}(L^s) \subseteq \mathrm{Min}(L) \cup \{\mathfrak{m}\}$ for all $s\geq 1$.  This proves (ii). 
 
 We now show that $L$ has the persistence property. Note that, in view of   the proof of  \cite[Corollary 4.10]{NQKR},   every normal monomial ideal has the  persistence property. Hence, it is sufficient for us to demonstrate  that $L$ is normal. Our strategy is to use Theorem \ref{I+xH}. To accomplish this, 
 set $$A:=\bigcap_{i=1}^{n-3}(x_i, x_{i+1}, x_{i+2}) ~ \text{and} ~  B:=(x_{n-1}, x_{n-2})\cap (x_{n-1}, x_1) \cap (x_1, x_2).$$
      Then we have $L=A\cap (x_nR+ B)=x_nA+A\cap B$. Due to  $A$ is the Alexander dual of the $3$-path ideal  of   a  path  graph $P=(V(P), E(P))$  with the vertex set  $V(P)=\{x_1, \ldots, x_{n-1}\}$ and the following edge set $$E(P)=\{\{x_i, x_{i+1}\}~:~ i=1, \ldots, n-2\},$$   Lemma 
      \ref{NTF-Alexander Dual-Path}  gives that $A$ is normally torsion-free, and by  \cite[Theorem 1.4.6]{HH1}, we deduce that $A$ is normal. Also, note that $A+A\cap B=A$ is normal. In what follows, our aim is to establish $A\cap B$ is normal.   It is easy  to check that
 \begin{align*} 
 A\cap B& = \bigcap_{i=2}^{n-4}(x_i, x_{i+1}, x_{i+2})\cap (x_{n-1}, x_{n-2})\cap (x_{n-1}, x_1) \cap (x_1, x_2)  \\
  & =    x_1(\bigcap_{i=2}^{n-4}(x_i, x_{i+1}, x_{i+2})\cap (x_{n-1}, x_{n-2}))   \\
  &\quad     + x_2R \cap   \bigcap_{i=3}^{n-4}(x_i, x_{i+1}, x_{i+2})\cap (x_{n-1}, x_{n-2})\cap (x_{n-1}, x_1). 
 \end{align*}
  To simplify the notation, put $C:=\bigcap_{i=2}^{n-4}(x_i, x_{i+1}, x_{i+2})\cap (x_{n-1}, x_{n-2})$ and 
 $$D:=x_2R \cap   \bigcap_{i=3}^{n-4}(x_i, x_{i+1}, x_{i+2})\cap (x_{n-1}, x_{n-2})\cap (x_{n-1}, x_1).$$ 
 This gives rise to  $A\cap B=x_1C+D$.  As   the monomial ideal $\bigcap_{i=2}^{n-4}(x_i, x_{i+1}, x_{i+2})$ can be viewed as 
 the Alexander dual of the  $3$-path ideal  of  a  path  graph,  by  Lemma \ref{NTF-Alexander Dual-Path}, is normally torsion-free, and on account of 
 Lemma  \ref{NQBM}, we can conclude that $C$ is normally torsion-free, and based on   \cite[Theorem 1.4.6]{HH1}, is normal. Moreover, since $D \subseteq C$, one has $C+D=C$ is normal as well. Now, a similar argument shows that the monomial ideal $\bigcap_{i=3}^{n-4}(x_i, x_{i+1}, x_{i+2})$ is normally torsion-free, and after two times using  Lemma  \ref{NQBM}, we can derive that $\bigcap_{i=3}^{n-4}(x_i, x_{i+1}, x_{i+2})\cap (x_{n-1}, x_{n-2})\cap (x_{n-1}, x_1)$ is normally torsion-free. It follows now from 
 \cite[Lemma 2.5]{ANKRQ} that $D$ is normal.  Thanks to  Theorem \ref{I+xH}, we get  $A\cap B= x_1C+D$ is normal. Due to $A$ and  $A\cap B$ are  normal, Theorem  \ref{I+xH} yields that  $L=x_nA+A\cap B$ is normal, and so has the persistence property, as we claimed in  (i).
 
 Finally, combining (i) and (ii) leads us to    $$\mathrm{Ass}^{\infty }(DI(C_n))\subseteq \{(x_i, x_{i+1}, x_{i+2})~:~i=1, \ldots, n\}  \cup \{\mathfrak{m}\}.$$ This demonstrates (iii), and the proof is complete.  
 \end{proof}


Let $G$ be a simple finite graph. We finish this paper with a short argument on the relation between  prime monomial ideals  associated to  $\mathrm{Ass}(DI(G)^s)$ for some $s \geq 1$, and prime monomial ideals associated to  $\mathrm{Ass}(DI(G_\mathfrak{p})^s)$, where 
 $G_\mathfrak{p}$ stands for the induced graph on $\mathfrak{p}$, i.e., the graph with vertex set  $\mathfrak{p}$, and edge set 
 $E({G_{\mathfrak{p}}})=\{e\in E(G) \mid e\subseteq \mathfrak{p}\}$.

Francisco, Ha, and Van Tuyl, in \cite[Lemma 2.11]{FHV2}, showed that if $\mathcal{H}$ is a finite simple hypergraph on $V=\{x_1, \ldots, x_n\}$ with cover ideal $J(\mathcal{H})\subseteq R=k[x_1, \ldots, x_n]$, then for all $d\geq 1$, 
  $$P=(x_{i_1}, \ldots, x_{i_r})\in \mathrm{Ass}(R/J(\mathcal{H})^d) \Leftrightarrow  P=(x_{i_1}, \ldots, x_{i_r})\in \mathrm{Ass}(k[P]/J(\mathcal{H}_P)^d),$$
  where $k[P]=k[x_{i_1}, \ldots, x_{i_r}]$, and  $\mathcal{H}_P$ is the induced hypergraph of $\mathcal{H}$ on the vertex set $P=\{x_{i_1}, \ldots, x_{i_r}\}\subseteq V$. Here,  we provide a counterexample which shows that one cannot mimic this result for dominating ideals of simple finite graphs. To do this, let $G=(V(G), E(G))$ be a simple finite graph with the vertex set $V(G)=\{x_1, x_2, x_3, x_4, x_5, x_6\}$ and the following edge set 
\begin{align*}
  E(G)= & \{\{x_1, x_2\}, \{x_1, x_3\}, \{x_1, x_6\}, \{x_2, x_4\},\{x_2,x_5\}, \{x_3,x_4\}, \\ 
  &\{x_3,x_5\}, \{x_4,x_6\}\}.
  \end{align*}
  One can easily see that the dominating ideal of $G$ is given by 
  \begin{align*}
  DI(G)= & (x_1,x_2,x_3,x_6) \cap (x_1,x_2,x_4,x_5)\cap (x_1,x_3,x_4,x_5)\cap (x_2,x_3,x_4,x_6)\\ 
  &\cap (x_2,x_3,x_5)\cap (x_1,x_4,x_6).
  \end{align*}
  Take  $\mathfrak{p}:=(x_1,x_2,x_3,x_6) \in \mathrm{Ass}(DI(G)^2)$. It follows from the definition that $V(G_\mathfrak{p})=\{x_1,x_2,x_3,x_6\}$ and 
  $E(G_\mathfrak{p})=\{\{x_1, x_2\}, \{x_1, x_3\}, \{x_1, x_6\}\}$. In particular, the dominating ideal of $G_\mathfrak{p}$ is given by 
  $$DI(G_\mathfrak{p})= (x_1,x_2) \cap (x_1,x_3)\cap (x_1,x_6)=(x_1, x_2x_3x_6).$$ 
  We can conclude from \cite[Theorem 2.5]{SN} that $DI(G_\mathfrak{p})$ is normally torsion-free, and so 
  $\mathfrak{p} \notin \mathrm{Ass}(DI(G_\mathfrak{p})^2)$. On the other hand, $\mathfrak{q}:=(x_1,x_6)\in \mathrm{Ass}(DI(G_\mathfrak{p})^2)$, while $\mathfrak{q} \notin \mathrm{Ass}(DI(G)^2)$.  

In what follows, we turn our attention to discuss the nature of  $G_{\mathfrak{p}}$. To accomplish this, we first recall the definition of  partial $t$-cover ideals which  has already been introduced  in \cite{BBV}. 
\begin{definition} (\cite[Definition 1.1]{BBV})  
Suppose that $G$ is a finite simple
graph on the vertex set $V(G) = \{x_1, x_2, \ldots, x_n\}$
 with the edge set $E(G)$. Also, for any $x\in V(G)$,  let
$N(x) = \{y~:~ \{x, y\}\in  E(G)\}$ denote the set of neighbors of $x$. Fix an integer $t\geq 1$. The {\it partial $t$-cover ideal} of $G$  is the monomial ideal
$$J_t(G)=\bigcap_{x\in V(G)}\left(\bigcap_{\{x_{i_1}, \ldots, x_{i_t}\} \subseteq N(x)}(x, x_{i_1}, \ldots, x_{i_t})\right).$$
\end{definition}

It should be noted that when $t=1$, the  construction above  is simply the cover ideal of a finite simple graph $G$. 
In fact, Bhat, Biermann, and Van Tuyl, in  \cite[Lemma 2.4]{BBV}, proved the following lemma:

\begin{lemma} (\cite[Lemma 2.4]{BBV})
Let $G$ be a graph on the vertex set  $\{x_1, \ldots, x_n\}$, and $J_t(G)$ be the partial $t$-cover ideal  of $G$. If  $\mathfrak{p}=(x_{i_1}, \ldots, x_{i_r})\in \mathrm{Ass}(J_t(G)^s)$, then $G_\mathfrak{p}$ is connected. 
\end{lemma}
In this direction,  we leave the following open question: 
\begin{question}\label{OPEN-QUESTION}
 Let $G$ be a connected simple finite graph  on the vertex set $V(G)=\{x_1, \ldots, x_n\}$, and let $DI(G)$ be the  dominating  ideal of $G$. If  
 $\mathfrak{p}=(x_{i_1}, \ldots, x_{i_r})\in \mathrm{Ass}(DI(G)^s)$ for some $s\geq 1$, then can we deduce that $G_\mathfrak{p}$ is connected?
\end{question}




\end{document}